\documentclass[a4paper,11pt]{article}
\usepackage{amsmath,amsthm, authblk}
\usepackage{amssymb,latexsym}
\usepackage{graphicx,subfig}
\usepackage{epsfig}
\usepackage{color}

\usepackage{lineno}
\newcommand*\patchAmsMathEnvironmentForLineno[1]{%
  \expandafter\let\csname old#1\expandafter\endcsname\csname #1\endcsname
  \expandafter\let\csname oldend#1\expandafter\endcsname\csname end#1\endcsname
  \renewenvironment{#1}%
     {\linenomath\csname old#1\endcsname}%
     {\csname oldend#1\endcsname\endlinenomath}}%
\newcommand*\patchBothAmsMathEnvironmentsForLineno[1]{%
  \patchAmsMathEnvironmentForLineno{#1}%
  \patchAmsMathEnvironmentForLineno{#1*}}%
\AtBeginDocument{%
\patchBothAmsMathEnvironmentsForLineno{equation}%
\patchBothAmsMathEnvironmentsForLineno{align}%
\patchBothAmsMathEnvironmentsForLineno{flalign}%
\patchBothAmsMathEnvironmentsForLineno{alignat}%
\patchBothAmsMathEnvironmentsForLineno{gather}%
\patchBothAmsMathEnvironmentsForLineno{multline}%
}

\newtheorem{thm}{Theorem}[section]

\newtheorem{clm}[thm]{Claim}

\title{A note on rainbow matchings in properly edge-coloured graphs
}
\author{Allan Lo}
\affil{School of Mathematics, University of Birmingham, \\Birmingham, B15 2TT, UK\\
\texttt{s.a.lo@bham.ac.uk}}

\begin{document}

\maketitle

\abstract{
A rainbow matching in an edge-coloured graph is a matching such that its edges have distinct colours. 
We show that every properly edge-coloured graph $G$ with $|G| \ge  (9\delta(G) -5)/2$ has a rainbow matching of size~$\delta(G)$, improving a result of Diemunsch~et al.
}

\section{Introduction}
Let $G = (V,E)$ be a simple undirected graph without loops. 
Write $|G|$, $\delta(G)$ and $\Delta(G)$ for the order, minimum degree and maximum degree of $G$ respectively.
A \emph{proper edge-colouring} of $G$ is a function $c:E \rightarrow \{1,2, \dots\}$ such that any two adjacent edges have distinct colours. 
If $G$ is assigned such a colouring $c$, then we say that $G$ is a \emph{properly edge-coloured graph}.
Denote the colour of the edge $e \in E$ by~$c(e)$.
A subgraph $H$ of $G$ is \emph{rainbow} if its edges have distinct colours. 
The study of rainbow matchings began with a conjecture of Ryser~\cite{ryser1967neuere}, which states that every Latin square of odd order contains a Latin transversal.
An equivalent statement is that for $n$ odd, every $n$-edge-coloured of complete bipartite graph $K_{n,n}$ contains a rainbow perfect matching. 
A survey on rainbow matchings and other rainbow subgraphs in edge-coloured graphs appears in~\cite{MR2438857}.

LeSaulnier et al.~\cite{MR2651735} proved that if $G$ is a properly edge-coloured graph with $G \ne K_4$ or $|G| \ne \delta(G) +2$, then $G$ contains a rainbow matching of size $\lceil \delta(G) /2 \rceil$.
If we further impose that $|G| \ge 8 \delta(G) /5$, then Wang~\cite{Wang2011} showed that $G$ contains a rainbow matching of size $\lfloor 3\delta(G) /5 \rfloor$.
In the same paper, Wang asked whether there exists a function $f(n)$ such that every properly edge-coloured graph~$G$ with $|G| \ge f(\delta(G))$ contains a rainbow matching of size~$\delta(G)$.
Clearly, if $f(n)$ exists, then $f(n) \ge 2n$.
In fact,  $f(n) > 2n$ for $n$ even as there exist $n \times n$ Latin square that have no Latin transversal (see \cite{MR1130611} and \cite{wanless2009transversals}).
Diemunsch et al.~\cite{diemunsch2011rainbow} gave an affirmative answer to Wang's question and showed that $f(n) =  \lfloor 13 n/2 - 23/2+41/(8n) \rfloor +1$ suffices.
In this article, we show that $f(n) = (9 n-5)/2$ would also be sufficient, improving the values of $f(n)$ for $n \ge 5$.

\begin{thm} \label{thm:main}
Every properly edge-coloured graph $G$ with $|G| \ge  (9\delta(G) -5)/2$ has a rainbow matching of size~$\delta(G)$.
\end{thm}

\section{Proof of Theorem~\ref{thm:main}}
Let $G$ be a properly edge-coloured graph with minimum degree $\delta$ and $n = |G| \ge (9\delta - 5)/2$ vertices.
The theorem trivially holds for $\delta = 1$, so we may assume that $\delta \ge 2$.
Suppose the theorem is false.
Let $G$ be a counterexample with $\delta$ minimal, so $n \ge (9\delta-5)/2$ and $G$ does not contain a rainbow matching of size~$\delta$.
We break down the proof into a series of simple claims.

\begin{clm} \label{clm:Delta}
$\Delta(G) \le 3(\delta -1)$. 
\end{clm}
\begin{proof}
Let $v$ be a vertices in $G$ such that $d(v) > 3(\delta-1)$.
By the minimal counterexample, there is a rainbow matching $M$ of size $\delta -1$ in $G \backslash \{v\}$.
Recall that $G$ is properly edge-coloured and $d(v) > 3(\delta-1)$, so there exists a vertex $u \in V(G) \backslash (V(M) \cup \{v\})$ such that the colour $c(uv)$ does not appear in $M$.
Thus, $M \cup \{uv\}$ is a rainbow matching of size $\delta$, a contradiction.
\end{proof}

Let $a$ be the size of the largest monochromatic matching $M_0$ in~$G$.
First we show that $a \ge 2$.
\begin{clm} \label{clm:a}
$a \ge 2$.
\end{clm}

\begin{proof}
If $a = 1$, then $G$ is rainbow.
Thus, it is enough to show that $G$ contains a matching of size~$\delta$.
Let $M= \{ x_i y_i : 1 \le i \le m \}$ be a matching of maximal size.
If $m \ge \delta$, then we are done.
For each $v \notin V(M)$, the neighbourhood of $v$ must lie in~$V(M)$.
Since $\delta(G) \ge \delta$, there exists an integer $1 \le i \le m$ such that both $x_iv$ and $y_iv$ are edges.
As $|G| \ge 3 \delta$, there exist an integer $1 \le i \le t$ and vertices $v,v' \notin V(M)$ such that $v \ne v'$ and both $x_iv$ and $y_iv'$ are edges.
Thus, $M \cup \{x_iv, y_iv'\} \backslash \{x_iy_i\}$ is a matching of size $m+1$ contradicting the maximality of~$M$.
\end{proof}

Fix a monochromatic matching $M_0$ of size~$a$.
By the minimal counterexample, there exists a rainbow matching $M= \{ x_i y_i : 1 \le i \le \delta -1 \}$ of size~$\delta-1$ in $G -M_0$.
Without loss of generality, we may assume that $c(x_iy_i) = i $ for $1 \le i \le \delta - 1$ and $M_0$ is of colour~$\delta$.
Note that every edge in $M_0$ intersects with $V(M)$ or else we can enlarge $M$ to a rainbow matching of size~$\delta$.
Set $W = V(G) \backslash V(M)$, so $|W| = n - 2(\delta -1)$.
We say that an edge $uv$ is \emph{good} if its colour is not in~$\{1, \dots, \delta-1\}$ and one of its vertices is in~$W$.
If there exists a good edge $uv$ in $G[W]$ (the induced edge-coloured subgraph of $G$ on $W$), then $M \cup \{uv\}$ is a rainbow matching of size~$\delta$.
Thus, we may assume that every good edge is incident with~$V(M)$, so every good edge lies between $V(M)$ and~$W$.

\begin{clm} \label{clm:goodvertex}
For $1 \le i \le \delta-1$, if $x_i$ is incident with at least three good edges, then no good edge is incident with~$y_i$, and vice versa.
\end{clm}

\begin{proof}
Suppose the contrary, so $x_i$ is incident with at least three good edges and $y_iu$ is a good edge.
Since $x_i$ is incident with at least three good edges, there exists $w \in W$ such that $c(x_iw) \ne c(y_iu)$ and $u \ne w$.
Then $M \cup \{x_iw, y_iu\} \backslash \{x_iy_i\}$ is a rainbow matching of size $\delta$, a contradiction.
\end{proof}

A vertex $v \in V(M)$ is \emph{good} if $v$ is incident with at least seven good edges.
By Claim~\ref{clm:goodvertex}, we may assume without loss of generality that $\{x_1,  \dots, x_r\}$ is the set of good vertices.
Let $W' = W \cup \{ y_1, \dots, y_r\}$.

\begin{clm} \label{clm:goodedge}
No edge $uv$ in $G[W']$ has colour in $\{1, \dots, r\}$.
\end{clm}

\begin{proof}
Suppose the contrary, so we may assume that there is an edge $uv$ in $G[W']$ such that $c(uv) =1$.
Since $G$ is properly edge-coloured, $v \ne y_1 \ne u$.
If $u,v \in W$, then there exists a vertex $w \in W$ such that $x_1w$ is good and $u \ne w \ne v$ and so $M \cup \{uv,x_1w\} \backslash \{x_1y_1\}$ is a rainbow matching of size~$\delta$, a contradiction.
Next assume that $u \in W$ and $v \notin W$, so $r \ge 2$ and $v = y_2$ say.
There exist distinct vertices $w_1,w_2 \in W \backslash \{u \}$ such that $x_1w_1$ and $x_2y_2$ are good edges with $c(x_1w_1) \ne c(x_2y_2)$.
Then, $M \cup \{y_2u, x_1w_1, x_2w_2\} \backslash \{x_1y_1, x_2y_2\}$ is a rainbow matching of size~$\delta$, a contradiction.
Finally, assume that $u,v \in W$ so $r \ge 3$ and $u = y_2$ and $v = y_3$ say.
Since $x_1$, $x_2$ and $x_3$ are good, there exist distinct vertices $w_1,w_2,w_3 \in W$  such that $x_1w_1$, $x_2w_2$ and $x_3w_3$ are good edges with distinct colours.
Thus, $M \cup \{y_2y_3, x_1w_1, x_2w_2, x_3w_3\} \backslash \{x_1y_1, x_2y_2, x_3y_3\}$ is a rainbow matching of size $\delta$, a contradiction.
\end{proof}

We say that an edge $uv$ is \emph{nice} if its colour is not in~$\{r+1, \dots, \delta-1\}$ and one of its vertices is in $W'$.
Note that every good edge is nice.
Recall that every good edge is incident with~$V(M)$.
By Claim~\ref{clm:goodvertex} and Claim~\ref{clm:goodedge}, no nice edge lies in $G[W']$.
Hence, every nice edge lies between $W'$ and $V(G) \backslash W'$.
A vertex $v \in V(M) \backslash \{ x_1,\dots,x_r,y_1, \dots,y_r\}$ is \emph{nice} if $v$ is incident with at least seven nice edges.
Note that if there is no good vertex i.e. $r=0$, the definition of good and nice vertex are the same and so there is also no nice vertex.
Next, we show the analogue of Claim~\ref{clm:goodvertex} and Claim~\ref{clm:goodedge} for nice vertices and edges.

\begin{clm} \label{clm:nicevertex}
For $r+1 \le i \le \delta-1$, if $x_i$ is incident with at least three nice edges, then no nice edge is incident with~$y_i$, and vice versa.
\end{clm}

\begin{proof}
Suppose the contrary, so $x_i$ is incident with at least three nice edges and $y_iu$ is a nice edge for some $r+1 \le i \le \delta-1$.
Here, we only consider one particular case as each remaining case can be verified using a similar argument.
Suppose that $r \ge 4$, $u=y_1$ and $c(y_iy_1) = 2$.
Since $x_i$ is nice, there exists a vertex $v \in W'$ such that $x_iv$ is nice and $v \ne u=y_1$ and $c(x_iv) \ne c(y_iy_1) = 2$.
Assume that $v = y_3$ and $c(x_iy_3 )=4$ .
Recall that $x_1, \dots, x_4$ are good vertices, so each is joined to at least seven good edges.
Thus, there exist distinct vertices $w_1,w_2,w_3,w_4 \in W$ such that $x_1w_1$, $x_2w_2$, $x_3w_3$ and $x_4w_4$ are good edges with distinct colours.
Therefore, $M \cup \{y_iy_1, x_iy_3 ,x_1w_1, x_2w_2, x_3w_3,x_4w_4\} \backslash \{x_1y_1, x_2y_2, x_3y_3,x_4y_4,x_iy_i\}$ is a rainbow matching of size $\delta$, a contradiction.
\end{proof}

By Claim~\ref{clm:nicevertex}, we may assume that $\{x_{r+1}, x_{r+2}, \dots, x_{r+s}\}$ is the set of nice vertices.

\begin{clm} \label{clm:niceedge}
No edge $uv$ in $G[W']$ has colour in $\{1, \dots, r+s\}$.
\end{clm}

\begin{proof}[Proof of claim]
By Claim~\ref{clm:goodedge}, the claim holds if $s =0$.
Assume that $s \ge 1$, so $r \ge 1$.
Suppose that there is an edge $uv$ in $G[W']$ with $c(uw) =r+1$ say.
Again, we only consider one particular case as each remaining case can be verified using a similar argument.
Suppose that $r \ge 4$, $u=y_1$ and $v = y_2$.
Since $x_{r+1}$ is nice, there exists a vertex $v' \in W'$ such that $v'x_{r+1}$ is a nice edge.
Assume that $v = y_3$ and $c(x_{r+1}y_3 )=4$.
Recall that $x_1, \dots, x_4$ are good vertices, so each is joined to at least seven good edges.
Thus, there exist distinct vertices $w_1,w_2,w_3,w_4 \in W$ such that $x_1w_1$, $x_2w_2$, $x_3w_3$ and $x_4w_4$ are good edges with distinct colours.
Therefore, $M \cup \{y_1y_2, x_{r+1}y_3 ,x_1w_1, x_2w_2, x_3w_3,x_4w_4\} \backslash \{x_1y_1, x_2y_2, x_3y_3,x_4y_4,x_{r+1}y_{r+1}\}$ is a rainbow matching of size $\delta$, a contradiction.
\end{proof}

Next, we counts the number of nice edges in $G'$.

\begin{clm} \label{clm:maxniceedges}
There are at most $(3\delta -9+s)r + 6(\delta - 1) $ nice edges.
\end{clm}

\begin{proof}
Recall that $V \backslash W' = \{x_1, \dots, x_{\delta-1}, y_{r+1}, \dots, y_{\delta-1} \}$ and every nice edge lies between $W'$ and $V \backslash W'$.
For $1 \le i \le r$, $x_i$ is joined to at most $3(\delta-1)$ nice edges as $\Delta(G) \le 3(\delta-1)$.
By Claim~\ref{clm:nicevertex} and the definition of nice, for $r+1 \le i \le r+s$ there are at most $r+6$ nice edges joining to $x_i$ and none to $y_i$.
For $r+s+1 \le i \le \delta -1$, there are at most six nice edges joining to $x_i$ or $y_i$ by Claim~\ref{clm:nicevertex}.
Therefore, the number of nice edges is at most
\begin{align}
3(\delta -1)r + (r+6)s + 6(\delta - 1-r-s) = (3\delta -9+s)r + 6(\delta - 1). \nonumber
\end{align}
\end{proof}

Recall that $V(M)$ is incident with $a$ edges of colour~$\delta$.
Hence, there are at least $2(a - \delta+1)$ vertices $v \in V(M)$ such that $c(vw) = \delta$ for some $w \in W$.
Let $t$ be the number of integers $i$ such that $r+s+1 \le i \le \delta-1$, and $c(x_iw) = \delta$ or $c(y_iw) = \delta$ for some $w \in W$.
Without loss of generality, we may assume that $r+s+1, \dots, r+s+t$ are such~$i$.
By Claim~\ref{clm:goodvertex} and Claim~\ref{clm:nicevertex}, we have 
\begin{align}
\textrm{$t \ge a - \delta+1 - \frac{r+s}{2}$ and $r+s+t \le \delta -1$.} \label{eqn:t}
\end{align}

\begin{clm} \label{clm:t}
For $r+s+1\le i \le r+s+t$, there are at most one edge of colours~$i$ in $G[W]$.
\end{clm}

\begin{proof}[Proof of claim]
Suppose $uv$ and $u'v'$ are edges of colour $i$ in $G[W]$ for some $r+s+1\le i \le r+s+t$.
Without loss of generality, we may assume that there exists $w \in W$ such that $c(x_{i}w) = \delta$ and $u \ne w \ne v$.
Hence, $M \cup \{uv, x_{i}w\} \backslash \{x_{i}y_{i}\}$ is a rainbow matching of size $\delta$, a contradiction.
\end{proof}

Now, we count the number of nice edges from $W'$ to $V \backslash W'$.
Recall that there are at most $a$ edges of the same colour.
By Claim~\ref{clm:niceedge}, there is no edge in $G[W']$ of colour~$r+1 \le i \le r+s$.
Thus, for $r+1 \le i \le r+s$ there are at most $a-1$ vertices in $W'$ that are incident with an edge of colour~$i$.
By Claim~\ref{clm:t}, for $r+s+1 \le i \le r+s+t$, there are at most $a$ vertices are $W$ that is incident with an edge of colour~$i$.
Recall that $W' \backslash W =\{ y_1, \dots, y_r\}$.
For $r+s+1 \le i \le r+s+t$, there are at most $a+r$ vertices in $W'$ that are incident with an edge of colour~$i$.
For $r+s+t+1 \le i \le \delta-1$, there are at most $2(a-1)$ vertices in $W'$ that is incident with an edge of colour~$i$.
Thus, the number of nice edges from $W'$ to $V \backslash W'$ is at least
\begin{align}
	& \delta|W'| - (a-1)s-(a+r)t -2(a-1)(\delta -1 - r-s-t)	 \nonumber\\
	= & \delta n - 2\delta(\delta-1) -  (a-1)(2\delta-2-2r-s)+(a-2)t +(\delta - t)r \nonumber \\
	\ge & \delta n - 2\delta(\delta-1) -  (a-1)(2\delta-2-2r-s)+(a-2)t +(r+s+1)r, \nonumber
\end{align}
where we recall that $|W'| = |W|+r = n - 2(\delta-1)+r$ and~\eqref{eqn:t}.
Since there are at most $(3\delta -9+s)r + 6(\delta - 1)$ nice edges in $G$ by Claim~\ref{clm:maxniceedges}, 
\begin{align}
	\delta n \le& (3\delta -10-r)r - (a-2)t + 2(\delta +3)(\delta - 1)  + (a-1)(2\delta-2-2r-s). \label{eqn:|G|}
\end{align}
For the remaining of the proof, we bound the right hand side of the above inequality from above to obtain a contradiction. 
Note that the coefficient of $t$ is $-(a-2) \le 0$ by Claim~\ref{clm:a}, so we can take the minimum value of $t$.
By~\eqref{eqn:t}, $t \ge a- \delta +1 -(r+s)/2$.

If $a \le \delta -1 +(r+s)/2$, then we take $t=0$.
The coefficient of $a$ becomes $2 \delta - 2- 2r -s \ge 2(\delta -1 - r-s) \ge 0$.
Thus, by taking $a = \delta -1 +(r+s)/2$, \eqref{eqn:|G|} becomes
\begin{align}
	\delta n \le 2(2 \delta+1) (\delta -1) + (2 \delta -7 -2r)r- (3r+s-2)s/2.\nonumber
\end{align}
Recall that if $s \ge 1$, then $r \ge 1$.
Hence, $(3r+s-2)s \ge 0$ and so 
\begin{align}
	\delta n \le & 2(2 \delta+1) (\delta -1) + (2 \delta -7 -2r)r
	\le   9\delta^2/2 -11\delta/2 +33/8 \nonumber
\end{align}
by taking $r= \delta/2-7/4$.
Hence, $n\le 9\delta/2 -11/2 +33/8\delta$, a contradiction.

If $a \ge \delta -1 +(r+s)/2$, we take $t = a- \delta +1 -(r+s)/2\ge 0$.
Then, \eqref{eqn:|G|} becomes
\begin{align}
	\delta |G| \le & (3\delta -1  - (3r+s)/2-a )a +(3 \delta -9 -r)r +2 (\delta+1)(\delta -1)
\end{align}
If $(3\delta -1)/2  - (3r+s)/4 \le \delta -1 +(r+s)/2$ , then right hand side is maximum when $a = \delta -1 +(r+s)/2$, which corresponds to the case when $a \le \delta -1 +(r+s)/2$ and so we are done.
Hence, we may assume that $(3\delta -1)/2  - (3r+s)/4 > \delta -1 +(r+s)/2$, so 
\begin{align}
	{5r+3s} < 2(\delta +1). \label{eqn:r}
\end{align}
Now we take $a = (3\delta -1)/2  - (3r+s)/4$, so 
\begin{align*}	
	\delta n \le & (3\delta -1 - (3r+s)/2)^2/4 +(3 \delta -7 +s)r +2 \delta(\delta -1)  \\
	= & \left( -\frac{7r}{16}+\frac{3\delta}4  - \frac{33}4\right)r + \left(\frac{3r}{8} +\frac{s}{16} -\frac{3\delta}4 + \frac14 \right)s + \frac{17 \delta^2}4 - \frac{3\delta}2 -\frac74	\\
	\le & \left( -\frac{7r}{16}+\frac{3\delta}4  - \frac{33}4\right)r + \left(\frac{3(5r+3s)}{40} -\frac{3\delta}4 + \frac14 \right)s + \frac{17 \delta^2}4 - \frac{3\delta}2 -\frac74	\\
	\le &  \left( -\frac{7r}{16}+\frac{3\delta}4  - \frac{33}4\right)r + \frac{(2-3\delta)s}{5} + \frac{17 \delta^2}4 - \frac{3\delta}2 -\frac74	\\
	\le & \left( -\frac{7r}{16}+\frac{3\delta}4  - \frac{33}4\right)r + \frac{17 \delta^2}4 - \frac{3\delta}2 -\frac74.
\end{align*}
Note that there is a maximal point at $r = 6(d-11)/7$.
Recall \eqref{eqn:r} that $r \le 2(\delta +2)/5$.
Therefore, 
\begin{align*}
n \le 
\begin{cases}
(17 \delta- 6)/4 - 7/(4\delta) & \textrm{if $\delta \le 11$}\\
(32 \delta - 60)/7 +260/(7\delta) & \textrm{if $12 \le \delta \le 22$}\\
112(\delta-1)/25 -863/(100\delta) & \textrm{if $\delta \ge 23$}
\end{cases}
\end{align*}
by taking $r=0$, $r = 6(d-11)/7$ and $r=2(\delta +2)/5$ respectively.
Moreover, $n< (9\delta - 5)/2$ a contradiction.
This completes the proof of Theorem~\ref{thm:main}.

\providecommand{\bysame}{\leavevmode\hbox to3em{\hrulefill}\thinspace}
\providecommand{\MR}{\relax\ifhmode\unskip\space\fi MR }
\providecommand{\MRhref}[2]{%
  \href{http://www.ams.org/mathscinet-getitem?mr=#1}{#2}
}
\providecommand{\href}[2]{#2}

\end{document}